\documentclass[12pt]{amsart}

\usepackage{tikz-cd}
\usepackage{amscd}
\usepackage{amsmath, amssymb, comment,esint}
\usepackage{amsfonts}
\usepackage{enumerate}
\usepackage{color}
\usepackage{url}

\newcommand{\Ric}{\mathrm{Ric}}
\newcommand{\ov}[1]{\overline{#1}}

\newcommand{\vol}{\mathrm{Vol}}
\newcommand{\Vol}{\mathrm{Vol}}
\newcommand{\diam}{\mathrm{diam}}
\newcommand{\ve}{\varepsilon}

\renewcommand{\leq}{\leqslant}
\renewcommand{\geq}{\geqslant}

\newcommand{\be}{\begin{equation}}
\newcommand{\ee}{\end{equation}}

\newcommand\blue{\color{blue}}

\newcommand{\HSC}{\mathrm{HSC}}
\newcommand{\BK}{\mathrm{BK}}
\newcommand{\CP}{\mathbb{CP}}

\begin{document}
\newcounter{remark}
\newcounter{theor}
\setcounter{theor}{1}
\newtheorem{claim}{Claim}
\newtheorem{theorem}{Theorem}[section]
\newtheorem{lemma}[theorem]{Lemma}
\newtheorem{corollary}[theorem]{Corollary}
\newtheorem{conjecture}[theorem]{Conjecture}
\newtheorem{proposition}[theorem]{Proposition}
\newtheorem{question}{Question}[section]
\newtheorem{definition}[theorem]{Definition}
\newtheorem{remark}[theorem]{Remark}

\numberwithin{equation}{section}

\title[On diameter rigidity for K\"ahler manifolds]{On diameter rigidity for K\"ahler manifolds with positive holomorphic sectional curvature}

\author{Jianchun Chu}
\address[Jianchun Chu]{School of Mathematical Sciences, Peking University, Yiheyuan Road 5, Beijing 100871, People's Republic of China}
\email{jianchunchu@math.pku.edu.cn}

\author{Man-Chun Lee}
\address[Man-Chun Lee]{Department of Mathematics, The Chinese University of Hong Kong, Shatin, Hong Kong, China}
\email{mclee@math.cuhk.edu.hk}

\author{Jintian Zhu}
\address[Jintian Zhu]{Institute for Theoretical Sciences, Westlake University, 600 Dunyu Road, 310030, Hangzhou, Zhejiang, People's Republic of China}
\email{zhujintian@westlake.edu.cn}

\renewcommand{\subjclassname}{\textup{2020} Mathematics Subject Classification}

\date{\today}

\begin{abstract}
In this paper, we establish some diameter rigidity for K\"ahler manifolds with positive holomorphic sectional curvature. In comparison with earlier works, our approach is based on constructing extreme $\mathbb{CP}^1$ through conformal deformation. 
\end{abstract}

\maketitle

\section{Introduction}

Comparison theorems are a fundamental tool for understanding the behavior of a manifold $(M^m,g)$ with a curvature lower bound. In the context of positive Ricci curvature, $\Ric(g)\geq (m-1)g$, the classical Bonnet–Myers theorem states that the diameter of $(M^m,g)$  is bounded above by that of the standard sphere $(\mathbb{S}^m,g_{\mathbb{S}^{m}})$, while Cheng’s diameter rigidity theorem shows that this optimal diameter estimate is attained only if $(M^m,g)$  is isometric to $(\mathbb{S}^m,g_{\mathbb{S}^{m}})$.

\begin{theorem}[Bonnet-Myers, Cheng]\label{Riemannian diameter comparison}
Let $(M,g)$ be a complete Riemannian manifold with real dimension $m$. If $\Ric(g)\geq(m-1)g$, then we have
\[
\diam(M^m,g) \leq \diam(\mathbb S^m,g_{\mathbb S^m}),
\]
and the equality holds if and only if $(M,g)$ is isometric to the standard sphere $(\mathbb S^m,g_{\mathbb S^m})$.
\end{theorem}

K\"ahler geometry is an important branch of mathematics, regarded as the intersection of Riemannian geometry, complex geometry, and symplectic geometry. The existence of a compatible complex structure imposes stronger constraints on the topology of the underlying manifold. In contrast to round spheres, which serve as the Riemannian space forms, the model space with positive curvature in the K\"ahler  category is the complex projective space $(\CP^n,\omega_{\CP^n})$ equipped with the Fubini–Study metric. This metric has constant holomorphic sectional curvature.

The understanding of comparison geometry in Riemannian geometry is fairly complete, whereas its K\"ahler counterpart remains mysterious and unclear. Indeed, there are various natural notions of curvature in K\"ahler geometry. The positivity of each notion of curvature encodes different aspects of the complex structure. It is both tempting and interesting to explore the geometric implications of different curvature constraints. Even in the case of positive Ricci curvature, basic comparisons such as volume and diameter appear to be more subtle than in the Riemannian setting.

Perhaps surprisingly, the K\"ahler analogue of volume comparison was established only recently by Zhang \cite{Zhang22}, using techniques arising from recent developments in K-stability theory, which are completely different from those used in traditional Riemannian geometry. The diameter comparison is even more subtle. In fact, Liu–Yuan observed \cite[Remark 1]{LY18} that the $n$-fold product $(\CP^1)^n:=\CP^1\times \cdots\times \CP^1$, equipped with the Kähler metric 
$$\omega=\frac{2}{n+1}(\omega_{\CP^1}+\cdots +\omega_{\CP^1})$$
shares the same Ricci curvature lower bound as $(\CP^n,\omega_{\CP^n})$ but has a larger diameter.  This failure, however, can be ruled out by imposing a stronger positivity condition in Kähler geometry: positive bisectional curvature, $\mathrm{BK}\geq 1$. To summarize, we recall

\begin{definition}
Let $(M,\omega)$ be a K\"ahler manifold and $u,v$ be two unitary $(1,0)$-vectors.
\begin{itemize}\setlength{\itemsep}{1mm}
\item Holomorphic bisectional curvature: $\BK(u,v)=R(u,\ov{u},v,\ov{v})$.
\item Holomorphic sectional curvature: $\HSC(u)=R(u,\ov{u},u,\ov{u})$.
\end{itemize}
Given any real constant $K$, we say that
\begin{itemize}\setlength{\itemsep}{1mm}
\item holomorphic bisectional curvature is bounded from below by $K$, denoted by $\BK\geq K$, if $\BK(u,v)\geq K(1+|\langle u,\ov{v}\rangle|^{2})$ for any unitary $(1,0)$-vectors $u,v$;
\item holomorphic sectional curvature is bounded from below by $2K$, denoted by $\HSC\geq 2K$, if $\HSC(u)\geq 2K$ for any unitary $(1,0)$-vector $u$.
\end{itemize}
\end{definition}

The case when the complex dimension is $n$ and $K=1$ corresponds to the Kähler space form  $(\CP^n,\omega_{\CP^n})$. Based on this model space form, Tsukamoto \cite{Tsukamoto57} first proved that the diameter of a compact K\"ahler manifold with $\HSC\geq 2$ (holds automatically when $\BK\geq 1$)  cannot exceed
\[
\diam(\CP^n,\omega_{\CP^n}) = \frac{\pi}{\sqrt 2}
\]
using a standard variational argument. In sharp contrast, rigidity is much more challenging, even under the stronger curvature condition $\mathrm{BK}\geq 1$. By the resolution of Frankel’s conjecture \cite{Frankel61}, independently by Siu–Yau \cite{SY80} via differential geometry and by Mori \cite{Mori79} via algebraic geometry (the more general Hartshorne conjecture was proved; see \cite{Hartshorne70}), $M^n$ must be biholomorphic to $\mathbb{CP}^n$. Building on this, rigidity was established only very recently by Datar–Seshadri \cite{DS23}, who introduced and used a monotonicity formula for a function arising from Lelong numbers of positive currents on  $\mathbb{CP}^n$; see also the earlier works of Liu–Yuan \cite{LY18} and Tam–Yu \cite{TY12}.

\begin{theorem}[Tsukamoto \cite{Tsukamoto57}, Datar-Seshadri \cite{DS23}]
Let $(M,\omega)$ be a compact $n$-dimensional K\"ahler manifold with $\BK\geq1$. Then we have
\[
\diam(M,\omega) \leq \frac{\pi}{\sqrt 2},
\]
and the equality holds if and only if $(M,\omega)$ is biholomorphically isometric to $(\mathbb{CP}^{n},\omega_{\mathbb{CP}^{n}})$.
\end{theorem}

In comparison with the condition considered (and used) by Tsukamoto \cite{Tsukamoto57}, Xiong–Yang asked whether diameter rigidity holds under the weaker positivity condition $\HSC\geq 2$.

\begin{question}[Xiong-Yang, {\cite[Problem 3.5]{XY24}}]\label{Xiong-Yang question}
Let $(M,\omega)$ be a compact $n$-dimensional K\"ahler manifold with $\HSC\geq2$. If
\[
\diam(M,\omega) = \frac{\pi}{\sqrt 2},
\]
then $(M,\omega)$ is biholomorphically isometric to $(\mathbb{CP}^n,\omega_{\mathbb{CP}^n})$.
\end{question}

{ 

Question~\ref{Xiong-Yang question} can be viewed as rigidity in both metric and complex structure. The  main goal of this work is to give partial progress toward Question~\ref{Xiong-Yang question}. Our first observation is that irreducibility is indeed necessary. In Appendix~\ref{sec:app}, we show that a product of scaled Fubini-Study metrics on $\mathbb{CP}^{m}\times\mathbb{CP}^{n-m}$ already satisfy the assumptions in Question~\ref{Xiong-Yang question}. 
Motivated by this, we first consider the rigidity problem on metrics on $\mathbb{CP}^n$, under some additional assumptions. This is also partly motivated by the works \cite{LY18, TY12}.
}
We prove an analogue of a result by Liu–Yuan \cite{LY18}, showing diameter rigidity when, in addition, the volume of $(\mathbb{CP}^n,\omega)$ is not an integer multiple of that of $(\mathbb{CP}^n, \omega_{\mathbb{CP}^n})$. In particular, we relax the curvature assumption in \cite{LY18}.

\begin{theorem}\label{Thm: main1}
Let $\omega$ be a K\"ahler metric on $\mathbb{CP}^{n}$ with $\HSC\geq2$. If
\[
\mathrm{diam}(\mathbb{CP}^{n},\omega) = \frac{\pi}{\sqrt 2}
\]
and
\[
\vol(\mathbb{CP}^{n},\omega) \neq k^{-n}\cdot\vol(\mathbb{CP}^n,\omega_{\mathbb{CP}^n}) \mbox{ for all integers }k\geq 2,
\]
then $(\mathbb{CP}^{n},\omega)$ is biholomorphically isometric to $(\mathbb{CP}^n,\omega_{\mathbb{CP}^n})$.
\end{theorem}

The example in Appendix~\ref{sec:app} also satisfies the volume assumption in Theorem~\ref{Thm: main1}, and so the requirement of irreducibility is necessary.

\medskip

Our second result is motivated by the work of Tam–Yu \cite{TY12}, who considered the scenario where the optimal diameter is achieved by two complex sub-manifolds instead of two points. In this case, Tam–Yu \cite{TY12} established rigidity under the condition $\mathrm{BK}\geq 1$, prior to the works of Liu–Yuan \cite{LY18} and Datar–Seshadri \cite{DS23}. In this regard, we provide a partial generalization of Tam–Yu’s diameter-type rigidity theorem in the case of  $\HSC\geq 2$.

\begin{theorem}\label{Thm: main2}
Let $(M,\omega)$ be a compact $n$-dimensional K\"ahler manifold with $\HSC\geq2$. If
\[
d(p,Q) = \frac{\pi}{\sqrt{2}},
\]
where $p$ is a point and $Q$ is a $(n-1)$-dimensional complex submanifold, then $(M,\omega)$ is biholomorphically isometric to $(\mathbb{CP}^n,\omega_{\mathbb{CP}^n})$.
\end{theorem}

The main ingredient to  Theorem \ref{Thm: main1} and Theorem \ref{Thm: main2} is the following existence result for an extreme $\CP^1$. This is inspired by the work of Liu-Yuan \cite{LY18}.
\begin{proposition}\label{totally gedesic CP1}
Let $(M,\omega)$ be a compact $n$-dimensional K\"ahler manifold with $\HSC\geq 2$ and $p,q$ be two points in $M$ with $d(p,q)=\pi/\sqrt{2}$. Then there is a complex curve $L$ of $M$ such that the following holds:
\begin{enumerate}\setlength{\itemsep}{1mm}
\item[(i)] $L$ is biholomorphic to $\CP^{1}$;
\item[(ii)] $L$ is totally geodesic and isometric to the sphere with radius $1/\sqrt{2}$, and $p,q$ are a pair of antipodal points on $L$.
\end{enumerate}
\end{proposition}

Such an existence result was first obtained by Liu–Yuan \cite{LY18} under the condition  $\mathrm{BK}\geq 1$, based on Tam–Yu’s complex Hessian comparison theorem \cite{TY12} and a strong maximum principle argument, where  $\mathrm{BK}\geq 1$ plays a prominent role. To extend this to a weaker curvature condition, we make use of the conformal perturbation method developed by the authors in \cite{CLZ25} to construct a family of minimizing geodesics attaining the extreme distance, whose union gives the extreme $\mathbb{CP}^1$ as desired. The method is flexible enough that it suffices to assume partial curvature positivity, $\mathrm{HSC}\geq 2$.

\bigskip

Apart from the diameter comparison theorem for K\"ahler manifolds, we point out that there are also many interesting comparison theorems concerning volume, eigenvalues, and other geometric quantities. We refer interested readers to \cite{LW05, Lott21, Zhang22, CWZ25, XY24, WY25, FXY25} and the references therein.

\bigskip

The rest of the paper is organized as follows. In Section \ref{Sec: extreme CP1}, we show how to use the conformal deformation to construct  an extreme $\CP^1$, establishing Proposition \ref{totally gedesic CP1}. In Section \ref{Sec: main theorems}, we use the existence of extreme $\mathbb{CP}^1$ to prove the diameter rigidity, i.e. Theorem \ref{Thm: main1} and Theorem \ref{Thm: main2}.

\section*{Acknowledgements}
The authors would like to thank the referee for useful comments. The first-named author was partially supported by National Key R\&D Program of China 2024YFA1014800 and 2023YFA1009900, and NSFC grants 12471052 and 12271008. The second-named author is supported by Hong Kong RGC grant (Early Career Scheme) of Hong Kong No. 24304222, No. 14300623 and No. 14304225, NSFC grant No. 12222122 and an Asian Young Scientist Fellowship. The third-named author was partially supported by National Key R\&D Program of China 2023YFA1009900, NSFC grant 12401072, and the start-up fund from Westlake University.

\section{ Construction of extreme $\mathbb{CP}^{1}$}\label{Sec: extreme CP1}

In this section, we will prove an extension of \cite[Theorem 1]{LY18}, i.e. Proposition~\ref{totally gedesic CP1}. Particularly, we will show that under $\mathrm{HSC}\geq 2$, if there is a pair of points $p,q\in M$ with distance $\pi/\sqrt{2}$, then the geodesic will be contained in some $\mathbb{CP}^1$ which is totally geodesic and embedded in $M$ both holomorphically and isometrically. We start with a lemma showing that along the geodesic which realizing the diameter upper bound enjoys a better curvature lower bound along the minimizing geodesic, using an observation of Yang \cite{Yang17}.

In this section, we will always assume $(M,\omega)$ is compact $n$-dimensional K\"ahler manifold with $\mathrm{HSC}\geq 2$. Denote its Riemannian metric and complex structure by $g$ and $J$.

\begin{lemma}\label{Lem: curvature rigidity}
If
$\gamma:[0,\pi/\sqrt{2\ell}]\to (M,g)$ is a unit-speed minimizing geodesic connecting $p$ and $q$ for some $0<\ell\leq 1$, then we have $\ell=1$ and
    $$\HSC(e_\gamma)\equiv 2$$ along the geodesic $\gamma$ where
\[
e_\gamma=\frac{1}{\sqrt 2}(\gamma'-\sqrt{-1}J\gamma').
\]
Furthermore, we have $$\BK(e_\gamma,u)\geq 1+|\langle e_\gamma,\bar u\rangle|^2$$ for any unitary $(1,0)$-vector $u$ and $t\in [0,\pi/\sqrt{2}]$.
\end{lemma}
\begin{proof}
 From the second variation formula of minimizing geodesics we have
\begin{equation}\label{eqn:sec-var-leng}
    \int_0^{\pi/\sqrt{2\ell}}|V'|^2-R(\gamma',V,V,\gamma')\,\mathrm ds\geq 0
\end{equation}
   for any smooth vector field $V$ along $\gamma$ satisfying $V(0)=V(\pi/\sqrt{2\ell})=0$. In particular, if we take
\[
 V(s)=\sin (\sqrt{2\ell}s)\cdot J\gamma',
\]
then  we have
\[
\begin{split}
    0&\leq \int_0^{\pi/\sqrt{2\ell}} 2\ell\cos^2(\sqrt{2\ell}s) -\sin^2 (\sqrt{2\ell}s) R(\gamma',J\gamma',J\gamma',\gamma') \,\mathrm ds\\
    &\leq  2\int_0^{\pi/\sqrt{2\ell}} \cos^2(\sqrt{2\ell}s) -\sin^2 (\sqrt{2\ell}s)  \,\mathrm ds = 0,
\end{split}
\]
where we use the fact that $\ell\leq 1$, $J\gamma'$ is parallel along $\gamma$ and
\[
R(\gamma',J\gamma',J\gamma',\gamma') = \HSC(e_{\gamma}) \geq 2.
\]
Then we have $\ell=1$ and $R(\gamma',J\gamma',J\gamma',\gamma')=2$ along $\gamma$. In particular, $\HSC$ attains its minimum at $e_\gamma$ among all unitary $(1,0)$-vectors. By Lemma \ref{partial BK}, we obtain $\BK(e_\gamma,u)\geq 1+|\langle e_\gamma,\bar u\rangle|^2$ for any unitary $(1,0)$-vector $u$.
\end{proof}

{
\begin{remark}\label{Lem: curvature rigidity remark}
By applying \eqref{eqn:sec-var-leng} to $V+sW$ at $s=0$, we see that the vector field 
\[
V(s)=\sin (\sqrt{2}s)\cdot J\gamma'
\]
is a Jacobi field along the geodesic $\gamma$.
\end{remark}
}

\bigskip

From now on, we fix $p,q\in M$ such that $d_g(p,q)=\pi/\sqrt{2}=\mathrm{diam}(M,g)$ which is realized by a minimizing geodesic $\gamma$ of unit speed. Then it follows from the work \cite{Ehr74}
that we can find a small $C^2$-neighborhood $\mathcal U$ of the metric $g$ and a small positive constant $i_0$ such that we have
$$i(M,h)\geq i_0>0$$
for any smooth metric $h\in\mathcal U$, where $i(M,h)$ denotes the injectivity radius of $(M,h)$. As a consequence, we can take an open neighborhood $V$ of $p$ such that the $h$-distance function $d_h(p,\cdot)$ to the point $p$ is smooth in $\mathring V:=V\setminus\{p\}$ for any smooth metric $h\in \mathcal U$. Without loss of generality, we can assume that $V$ is a $g$-geodesic ball centred at $p$ of  radius  $r<i(M,g)$. From the classical Hessian estimates for distance function in Riemannian geometry, by choosing $r$ sufficiently small, we can assume further that distance functions $d_h(p,\cdot)$ converge to $d_g(p,\cdot)$ in the sense of $C^1_{\mathrm{loc}}(\mathring V)$ whenever metrics $h\in\mathcal U$ converge to $g$ in $C^2$-sense. Here $h\in \mathcal{U}$ is not necessarily a K\"ahler metric with respect to the original complex structure $J$.

Fix a point $z_0$ near $p$ on the geodesic $\gamma$, which is also contained in $\mathring V$. Let us construct a special coordinate chart around $z_0$. Denote
$$X=\nabla^g d_g(p,\cdot) \mbox{ and }Y=JX\mbox{ in }\mathring V.$$
Take an embedded { (real)} hyper-surface $\Sigma$ passing through $z_0$ and perpendicular to $Y(z_0)$ at $z_0$. Let
$$\Phi:\mathring V\times (-\ve,\ve)\to M$$
denote the flow generated by $Y$. From our choice of $\Sigma$ it follows that the restricted map
$$\Phi:\Sigma\times (-\ve,\ve)\to M$$
is a diffeomorphism around $z_{0}$. By choosing a coordinate chart $\{x^i\}_{i=1}^{2n-1}$ on $\Sigma$  at $z_0$ where it  corresponds to the origin, we can obtain a coordinate chart $\{x^{1},\ldots,x^{2n-1},t\}$ of $M$ around $z_0$, where $z_0$ corresponds to the origin as well.

Given a small constant $r$ with $|r|\leq 1/2$ and a positive integer $k$, we  define
$$ \rho_{r,k}=\left[4r^2-k^2\sum_{i=1}^{2n-1} (x^i)^2-(t-r)^2\right]_+$$
and
$$\mathcal E_{r,k}=\{\rho_{r,k}>0\}$$
{ where $\mathcal{E}_{r,k}$ is a interior of an ellipsoid.}

We first show that the geodesic $\gamma$ can be perturbed in a way that it intersects $\overline{\mathcal E_{r,k}}\cap \{t\equiv r\}$ using method analogous to \cite{CLZ25}.

\begin{proposition}\label{Prop: rough passing geodesic}
There exists a minimizing geodesic $\bar\gamma_{r,k}$ connecting $p$ and $q$, which intersects  $\overline{\mathcal E_{r,k}}\cap \{t\equiv r\}$.
\end{proposition}

From now on, we fix the constants $r$ and $k$ where $|r|\leq 1/2$. In particular, we always have $0\leq \rho_{r,k}\leq 1$. Let
$f:(-\infty,+\infty)\to [0,+\infty)$ be the smooth function given by
\[
f(t) =
\begin{cases}
\ e^{-\frac{1}{t}} & \mbox{for $t>0$,} \\[1mm]
\ 0 & \mbox{for $t\leq 0$.}
\end{cases}
\]
Given any positive integer $\ell$, we consider the conformal deformation
$$g_{\ell}=e^{-2\psi_\ell}g\mbox{ with }\psi_\ell= f(\ell^{-1}\rho_{r,k}).$$
Denote $d_{\ell}:=d_{g_{\ell}}(p,q)$ and  let
$$\gamma_{\ell}:[0,d_{\ell}]\to (M,g_{\ell})$$
be a minimizing geodesic of unit speed connecting $p$ and $q$ with respect to the metric $g_{\ell}$.

For convenience, we always use $C$ to denote any constants independent of $\ell$, which may vary from line to line. For simplicity, we will use $\nabla$ to denote the Levi-Civita connection of the metric $g$. We first need some estimates in controlling the distortion of $g_{\ell}$ from being K\"ahler.

\begin{lemma}\label{Lem: tangent vector after conformal}
    We have
    $$|J\gamma_{\ell}'|_{g_{\ell}}\equiv 1\mbox{ and }|(J\gamma_{\ell}')'|_{g_{\ell}}\leq C(e^{\psi_\ell}|\nabla\psi_\ell|_g)\circ \gamma_{\ell}$$
    along the geodesic $\gamma_{l}$.
\end{lemma}
\begin{proof}
Since the tangent vector $\gamma_{\ell}'$ is of unit length with respect to $g_{\ell}$, the vector
$$v:=e^{-\psi_\ell}\cdot\gamma_{\ell}'$$
is of unit length with respect to $g$. Since $(M,g,J)$ is K\"ahler, we have
\[
g_{\ell}(J\gamma_{\ell}',J\gamma_{\ell}') = e^{-2\psi_l}g(J(e^{\psi_\ell} v),J(e^{\psi_\ell} v)) = g(Jv,Jv) = 1.
\]
On the other hand since $\gamma_\ell'$ is parallel with respect to $g_\ell$, we have
\[
 |\nabla^{g_\ell}_{\gamma_\ell'}(J\gamma_\ell')|_{g_\ell}
 =|(\nabla^{g_\ell}_{\gamma_\ell'}J )\gamma_\ell'|_{g_\ell}
 \leq  |\nabla^{g_{\ell}}J|_{g_{\ell}}\circ\gamma_\ell\leq C(e^{\psi_\ell}|\nabla\psi_\ell|_g)\circ\gamma_{l},
\]
where we have used Lemma \ref{quantity after conformal}. This completes the proof.
\end{proof}

We will often re-parametrize the time to mimick the situation in $g$.
\begin{lemma}\label{lma:beta}
   For all sufficiently large $\ell$, there is a constant $\beta(\ell)\in (0,1]$ such that
    $$ \int_0^{d_{\ell}}e^{\beta\psi_\ell\circ\gamma_{\ell}}\,\mathrm ds= \frac{\pi}{\sqrt{2}}.$$
\end{lemma}
\begin{proof}
    Notice that the integral is continuous with respect to $\beta$. When $\beta=0$, the left hand side equals to $d_\ell=d_{g_\ell}(p,q)$. Since $g_\ell\leq g$ and $g_\ell<g$ somewhere  along $\gamma$, we must have 
    \begin{equation}\label{eqn:distance comparison}
    d_{g_{\ell}}(p,q) \leq L_{g_{\ell}}(\gamma) < L_{g}(\gamma)
    = d_{g}(p,q) = \frac{\pi}{\sqrt{2}}.
    \end{equation}
    When $\beta=1$,
\[
\int^{d_\ell}_0 e^{\psi_\ell\circ \gamma_\ell} \,ds = \int^{d_\ell}_0 |\gamma_\ell'|_{g} \,ds=L_g(\gamma_\ell)\geq d_{g}(p,q)=\frac{\pi}{\sqrt{2}}
\]
since $\gamma_\ell$ connects $p$ and $q$. The result follows from the intermediate value theorem.
\end{proof}

In the following, we fix $\beta(\ell)$ from Lemma~\ref{lma:beta}, for every large $\ell$. We also define
\begin{equation}\label{eqn:defn-G}
G_\ell(t)=\int_0^te^{\beta\psi_\ell\circ \gamma_{\ell}}\,\mathrm ds,\;\mbox{ for }\,t\in [0,d_{\ell}].
\end{equation}
This is a re-parametrization of time in  suitable sense.

\begin{lemma}\label{Lem: error analysis}
We have
\[
\int_0^{d_{\ell}}\sin^2(\sqrt{2}G_\ell(s)) \left(\rho_{r,k}^{-4} \cdot \psi_\ell \cdot w_{\ell}(\rho_{r,k})^2\right)\circ\gamma_{\ell}\,\mathrm ds
\leq {}
 C\ell^{-1}\int_0^{d_{\ell}}\left(\rho_{r,k}^{-2} \cdot \psi_\ell\right)\circ\gamma_{\ell}\,\mathrm ds
\]
for $\ell$ large enough, where $w_\ell=J\gamma'_{\ell}/|\gamma'_{\ell}|_g$.
\end{lemma}

\begin{proof}
    From the second variation formula of minimizing geodesics, we have
    \begin{equation}\label{Eq: 001}
        \int_0^{d_{\ell}}|V'|_{g_{\ell}}^2-R^{g_{\ell}}(\gamma'_{\ell},V,V,\gamma'_{\ell})\,\mathrm ds\geq 0
    \end{equation}
    for any smooth vector field $V$ along $\gamma_{\ell}$ satisfying $V(0)=V(d_{\ell})=0$. If we take
\[
  V(s)=\sin(\sqrt{2}G_\ell(s))\cdot J\gamma_{\ell}',
\]
then $V(s)$ is an admissible vector field. By direct computation, we have
\begin{equation}\label{eqn:R-V}
\begin{cases}
\ V'(s)=\sqrt{2}\cos(\sqrt{2} G_\ell(s))\cdot e^{\beta\psi_\ell\circ\gamma_{\ell}}\cdot J\gamma_{\ell}'+\sin(\sqrt{2}G_\ell(s))\cdot (J\gamma_{\ell}')', \\[2mm]
\ R(\gamma_{\ell}',V,V,\gamma_{\ell}')=\sin^2(\sqrt{2}G_\ell(s))\cdot R(\gamma_{\ell}',J\gamma_{\ell}',J\gamma_{\ell}',\gamma_{\ell}').
\end{cases}
\end{equation}

Using Lemma \ref{Lem: tangent vector after conformal} and the fact that $\beta\in (0,1]$, we have
    \begin{equation}\label{Eq: 002}
    \begin{split}
        \int_0^{d_{\ell}}|V'(s)|_{g_\ell}^2\,\mathrm ds&\leq \int_0^{d_{\ell}} 2\cos^2(\sqrt{2}G_\ell(s)) \cdot e^{2 \psi_\ell\circ\gamma_{\ell}}\,\mathrm ds\\
        &\quad +C\int_0^{d_{\ell}} e^{2\psi_\ell\circ \gamma_{\ell}}(|\nabla\psi_\ell|_g+|\nabla\psi_\ell|_g^2)\circ \gamma_{\ell}\,\mathrm ds.
    \end{split}
    \end{equation}
Now for convenience, we denote
\[
v_\ell=\frac{\gamma_{\ell}'}{|\gamma_{\ell}'|_g}\mbox{ and }w_\ell=Jv_\ell.
\]
It is clear that  $|v_\ell|_g=|w_\ell|_g=1$ and $g(v_\ell,w_\ell)=0$. From Lemma \ref{quantity after conformal}, \eqref{eqn:R-V} and $\mathrm{HSC}(g)\geq 2$, we have
    \begin{equation}\label{Eq: 003}
    \begin{split}
   &\quad      \int_0^{d_{\ell}}R^{g_{\ell}}(\gamma'_{\ell},V,V,\gamma'_{\ell})\,\mathrm ds\\
   &= \int_0^{d_{\ell}}\sin^2 (\sqrt{2}G_\ell(s))\cdot R^{g_{\ell}}(\gamma'_{\ell},J\gamma_\ell',J\gamma_\ell',\gamma'_{\ell})\,\mathrm ds\\
        &\geq \int_0^{d_{\ell}}2e^{2\psi_\ell\circ\gamma_\ell}\sin^2(\sqrt{2}G_\ell(s))\,\mathrm ds\\
       &\quad + \int_0^{d_{\ell}}e^{2\psi_\ell\circ\gamma_\ell}\sin^2(\sqrt{2}G_\ell(s))\left[\nabla^2\psi_\ell(v_\ell,v_\ell)+\nabla^2\psi_\ell(w_\ell,w_\ell)\right]\circ\gamma_{\ell}\,\mathrm ds\\
        &\quad-C\int_0^{d_{\ell}}e^{2\psi_\ell\circ \gamma_{\ell}}|\nabla\psi_\ell|^2_g\circ \gamma_{\ell}\,\mathrm ds.
    \end{split}
    \end{equation}

By substituting \eqref{Eq: 002} and \eqref{Eq: 003} into \eqref{Eq: 001}, we deduce that
\begin{equation}\label{Eq: 004}
\begin{split}
& \int_0^{d_{\ell}}2\sin^2(\sqrt{2}G_\ell(s))\cdot e^{2\psi_\ell\circ\gamma_\ell}\,\mathrm ds \\
& +\int_0^{d_{\ell}}e^{2\psi_\ell\circ\gamma_\ell}\sin^2(\sqrt{2}G_\ell(s))
\left[\nabla^2\psi_\ell(v_\ell,v_\ell)+\nabla^2\psi_\ell(w_\ell,w_\ell)\right]\circ\gamma_{\ell}\,\mathrm ds \\
\leq {} & \int_0^{d_{\ell}} 2\cos^2(\sqrt{2}G_\ell(s)) \cdot e^{2 \psi_\ell\circ\gamma_{\ell}}\,\mathrm ds  +C\int_0^{d_{\ell}} e^{2\psi_\ell\circ \gamma_{\ell}}(|\nabla\psi_\ell|_g+|\nabla\psi_\ell|_g^2)\circ \gamma_{\ell}\,\mathrm ds.
\end{split}
\end{equation}

We now treat the last term. By direct computation and our choice of $\psi_\ell$, we have
\[
\nabla\psi_\ell=\ell \cdot \rho_{r,k}^{-2} \cdot \psi_\ell \cdot \nabla \rho_{r,k}
\]
and
\[
\nabla^2\psi_\ell= \ell \cdot \rho_{r,k}^{-2} \cdot \psi_{\ell} \cdot \nabla^2 \rho_{r,k}
+(\ell^2-2\ell\rho_{r,k}) \cdot \rho_{r,k}^{-4} \cdot \psi_{\ell} \cdot \nabla\rho_{r,k}\otimes\nabla \rho_{r,k}.
\]
In the local coordinate chart around $z$, we can always assume
\[
C^{-1}\delta_{ij}\leq g_{ij}\leq C\delta_{ij},\quad \mbox{ and }\quad |\partial g|\leq C.
\]
so that 
\begin{equation}\label{Eq: 005}
|\nabla\psi_\ell|_g\leq C \ell \cdot \rho_{r,k}^{-2} \cdot \psi_\ell
\end{equation}
and
\begin{equation}\label{Eq: 006}
\nabla^2\psi_\ell(e,e)\geq \frac12 \ell^2 \cdot \rho_{r,k}^{-4} \cdot \psi_\ell \cdot e(\rho_{r,k})^2-C \ell \cdot \rho_{r,k}^{-2} \cdot \psi_\ell
\end{equation}
for all $\ell$ large enough and any $g$-unit vector $e$. Substituting \eqref{Eq: 005}, \eqref{Eq: 006} into \eqref{Eq: 004} and using $1\leq e^{2\psi_\ell}\leq C$, we see that
\begin{equation}\label{Eq: 007}
\begin{split}
& \int_0^{d_{\ell}}\sin^2(\sqrt{2}G_\ell(s))\cdot e^{2\psi_\ell\circ\gamma_\ell}\,\mathrm ds \\
& +C^{-1}\int_0^{d_{\ell}}\sin^2(\sqrt{2}G_\ell(s)) \left( \ell^{2} \cdot \rho_{r,k}^{-4} \cdot \psi_\ell \cdot w_{\ell}(\rho_{r,k})^2\right)\circ\gamma_{\ell}\,\mathrm ds \\
\leq {} & \int_0^{d_{\ell}} \cos^2(\sqrt{2}G_\ell(s)) \cdot e^{2 \psi_\ell\circ\gamma_{\ell}}\,\mathrm ds \\
& +C\int_0^{d_{\ell}}\left(\ell \cdot \rho_{r,k}^{-2} \cdot \psi_\ell+\ell^{2} \cdot \rho_{r,k}^{-4} \cdot \psi_{\ell}^{2}\right)\circ\gamma_{\ell}\,\mathrm ds.
\end{split}
\end{equation}

From \eqref{eqn:defn-G}, we see that $G_{\ell}'=e^{\beta\psi_\ell\circ\gamma_\ell}$. Then by the variable substitution,
\[
\int_0^{d_{\ell}}\cos^2(\sqrt{2}G_\ell(s))\cdot e^{\beta\psi_\ell\circ\gamma_\ell}\,\mathrm ds
= \int_0^{d_{\ell}}\sin^2(\sqrt{2}G_\ell(s))\cdot e^{\beta\psi_\ell\circ\gamma_\ell}\,\mathrm ds.
\]
Together with
\[
|\psi_{\ell}|\leq C \mbox{ and } |e^{\beta\psi_\ell}-e^{2\psi_\ell}| \leq C\psi_{\ell},
\]
we have
\[
\int_0^{d_{\ell}}\cos^2(\sqrt{2}G_\ell(s))\cdot e^{2\psi_\ell\circ\gamma_\ell}\,\mathrm ds
\leq \int_0^{d_{\ell}}\sin^2(\sqrt{2}G_\ell(s))\cdot e^{2\psi_\ell\circ\gamma_\ell}+C\psi_{\ell}\circ\gamma_{\ell}\,\mathrm ds.
\]

Substituting this into \eqref{Eq: 007} yields
\begin{equation}\label{Eq: 008}
\begin{split}
& \int_0^{d_{\ell}}\sin^2(\sqrt{2}G_\ell(s)) \left( \ell^{2} \cdot \rho_{r,k}^{-4} \cdot \psi_\ell \cdot w_{\ell}(\rho_{r,k})^2\right)\circ\gamma_{\ell}\,\mathrm ds \\
\leq {}
& C\int_0^{d_{\ell}}\left(\ell \cdot \rho_{r,k}^{-2} \cdot \psi_\ell+\ell^{2} \cdot \rho_{r,k}^{-4} \cdot \psi_{\ell}^{2}+\psi_{\ell}\right)\circ\gamma_{\ell}\,\mathrm ds.
\end{split}
\end{equation}
Using $t^{-2}e^{-1/t}\leq C$ for $t\in(0,1]$,
\begin{equation}\label{Eq: 009}
\ell \cdot \rho_{r,k}^{-2} \cdot \psi_\ell+\ell^{2} \cdot \rho_{r,k}^{-4} \cdot \psi_{\ell}^{2}+\psi_{\ell} \leq C\ell \cdot \rho_{r,k}^{-2} \cdot \psi_\ell.
\end{equation}
Now Lemma \ref{Lem: error analysis} follows from \eqref{Eq: 008} and \eqref{Eq: 009}.
\end{proof}

In the following discussion, we always assume that the constant $r$ is small enough such that we have
$$d_g(\mathcal E_{r,k},\{p,q\})\geq 2\delta$$
for some constant $\delta>0$. Now, we are ready to prove Proposition \ref{Prop: rough passing geodesic}.

\begin{proof}[Proof of Proposition \ref{Prop: rough passing geodesic}]
Up to a subsequence, we can always assume that the geodesics $\gamma_\ell$ converge to a limit one, denoted by $\gamma_\infty$ as $\ell\to+\infty$. It is clear that $\gamma_\infty$ has length $\pi/\sqrt{2}$ with respect to $g$ and connects $p$ and $q$, thus it is a $g$-minimizing geodesic. Therefore, it remains to show that $\gamma_\infty$ has non-empty intersection with $ \overline{\mathcal E_{r,k}}\cap \{t\equiv r\}$.

{ By \eqref{eqn:distance comparison}, we see that
\[
L_{g_{\ell}}(\gamma_{\ell}) = d_{g_{\ell}}(p,q) < d_{g}(p,q) \leq L_{g}(\gamma_{\ell})
\]
}
so that $\gamma_\ell$ must intersect $\mathcal E_{r,k}$. Denote
$$I_\ell:=\gamma_\ell^{-1}(\mathcal E_{r,k})\neq \emptyset.$$
Suppose on the contrary that $\gamma_\infty$ does not intersect with $\overline{\mathcal E_{r,k}}\cap \{t\equiv r\}$, then we have
\[
\iota:=\inf_{\ell\geq \ell_0}\inf_{s\in I_\ell}|t-r|(\gamma_\ell(s))>0\mbox{ for some }\ell_0\in \mathbb N_+.
\]

Since the metrics $g_{\ell}$ converge to $g$ in $C^2$-sense as $\ell\to \infty$, we know that the gradients $\nabla^{g_{\ell}}d_{g_\ell}(p,\cdot)$ converge uniformly to $\nabla d_g(p,\cdot)$ in $C^0(\mathcal E_{r,k})$ as $\ell\to \infty$.  In particular, we have
\[
\begin{split}
w_{\ell}(\rho_{r,k})
& = \frac{1}{|\gamma_\ell'|_g} g(\nabla \rho_{r,k},J\gamma_\ell')
= \frac{1}{|\gamma_\ell'|_g} g(\nabla \rho_{r,k},J\nabla^{g_{\ell}}d_{g_\ell}) \\[2mm]
& \to g(\nabla \rho_{r,k}, J\nabla d_{g})=g(\nabla \rho_{r,k}, \partial_t) =\partial_t \rho_{r,k}=-2(t-r) \\[1mm]
\end{split}
\]
as $\ell\to+\infty$. 
   Thus, we might assume
    \begin{equation}\label{eqn:w_est}
    \begin{split}
    |w_\ell(\rho_{r,k})|\geq \iota
    \end{split}
    \end{equation}
for $\ell$ large enough, along $\gamma_\ell|_{I_\ell}$.

\smallskip

On the other hand, it follows from the $C^0$-convergence of $g_\ell$ to $g$ that we can have
    $$d_{g_{\ell}}(\mathcal E_{r,k},\{p,q\})\geq \delta$$
for all $\ell$ large enough. 

We note that $\gamma_\ell(0)=p$ and $\gamma_\ell(d_\ell)=q$. If $t\in I_\ell$, then $d_{g_\ell}(\gamma_\ell(t),p)\geq \delta$ implies $t\geq \delta$ so that
\begin{equation}\label{eqn:G-low}
G_\ell(t)=\int^t_0 e^{\beta \psi_\ell\circ \gamma_\ell} \, \mathrm ds\geq t \geq \delta.
\end{equation}

Similarly, $d_{g_\ell}(\gamma_\ell(t),q)\geq \delta$ implies  $t\leq d_\ell-\delta$ so that Lemma~\ref{lma:beta} implies
\begin{equation}\label{eqn:G-upp}
\begin{split}
G_\ell(t)&=\int^t_0 e^{\beta \psi_\ell\circ \gamma_\ell} ds\leq \int^{d_\ell-\delta}_0 e^{\beta \psi_\ell\circ \gamma_\ell}\, \mathrm ds\\
&=\frac{\pi}{\sqrt{2}}-\int^{d_\ell}_{d_\ell-\delta} e^{\beta \psi_\ell\circ \gamma_\ell}\, \mathrm ds\leq \frac{\pi}{\sqrt{2}}-\delta.
\end{split}
\end{equation}
Using \eqref{eqn:w_est}, \eqref{eqn:G-low} and \eqref{eqn:G-upp}, we compute
\begin{equation}\label{eqn:int-low}
\begin{split}
\int_0^{d_{\ell}}\sin^2(\sqrt{2}G_\ell(s)&) \left(\rho_{r,k}^{-4} \cdot \psi_\ell \cdot w_{\ell}(\rho_{r,k})^2\right)\circ\gamma_{\ell}\,\mathrm ds \\
& \geq \iota^2 \cdot \sin^2(\sqrt{2}\delta) \int_{I_\ell}(\rho_{r,k}^{-4} \cdot \psi_\ell)\circ\gamma_{\ell} \,\mathrm ds.
\end{split}
\end{equation}

Since $\rho_{r,k}\leq 1$ and $\psi_\ell$ is supported on $\mathcal{E}_{r,k}$, then
\begin{equation}\label{eqn:int-upp}
\ell^{-1}\int_0^{d_{\ell}}\left(\rho_{r,k}^{-2} \cdot \psi_\ell\right)\circ\gamma_{\ell}\,\mathrm ds
\leq C\ell^{-1}\int_{I_{\ell}}\left(\rho_{r,k}^{-4} \cdot \psi_\ell\right)\circ\gamma_{\ell}\,\mathrm ds.
\end{equation}
Combining \eqref{eqn:int-low} and \eqref{eqn:int-upp} with Lemma \ref{Lem: error analysis}, we obtain
\[
\iota^2 \cdot \sin^2(\sqrt{2}\delta) \int_{I_\ell}(\rho_{r,k}^{-4} \cdot \psi_\ell)\circ\gamma_{\ell} \,\mathrm ds
\leq C\ell^{-1}\int_{I_{\ell}}\left(\rho_{r,k}^{-4} \cdot \psi_\ell\right)\circ\gamma_{\ell}\,\mathrm ds.
\]
This is impossible when $\ell$ is large enough.
\end{proof}

Recall that $Y=J\nabla d_g(p,\cdot)$ is a smooth vector field in $\mathring V$  perpendicular to $X=\nabla d_g(p,\cdot)$, which means that the integral curves of $Y$ lie in a $g$-geodesic sphere centred at $p$. In particular, there is an integral curve
$$\xi:(-\infty,\infty)\to \mathring V$$
of the vector field $Y$ with $\xi(0)=z_0$.
\begin{lemma}\label{Lem: point pass geodesic}
    For any constant $r$, there is a minimizing geodesic connecting $p$ and $q$ passing through $\xi(r)$.
\end{lemma}
\begin{proof}
Define
$$\mathcal S=\{r\in\mathbb R:\mbox{ there is a minimizing geodesic passing through }\xi(r)\}.$$
Clearly, we have $\mathcal S\neq \emptyset$ since $0\in\mathcal S$, and $\mathcal S$ is closed. So it remains to show that $\mathcal S$ is also open. Without loss of generality, we only show that $\mathcal S$ is open at $0$. Namely for every small constant $r$, there is a minimizing geodesic passing through $\xi(r)$.

From Proposition \ref{Prop: rough passing geodesic} we can find a minimizing geodesic $\bar\gamma_{r,k}$ connecting $p$ and $q$, which intersects with $\overline{\mathcal E_{r,k}}\cap\{t\equiv r\}$. Up to a subsequence, we can assume that the geodesics $\bar\gamma_{r,k}$ converge to a limit minimizing geodesic connecting $p$ and $q$, denoted by $\bar\gamma_r$ as $k\to+\infty$. Notice that
$$\lim_{k\to\infty}\overline{\mathcal E_{r,k}}\cap \{t\equiv r\}=\{\xi(r)\}.$$
Therefore, we conclude that $\bar\gamma_r$ passes through the point $\xi(r)$. This completes the proof.
\end{proof}

It follows from Lemma \ref{Lem: point pass geodesic} that the integral curve $\xi$ always lies in the injective region with respect to $p$, we can define
$$\bar v(r)=\frac{\exp_p^{-1}(\xi(r))}{|\exp_p^{-1}(\xi(r))|_g}.$$
From previous discussion we have shown that the curves
$$\bar\gamma_r(s):=\exp_p(s\bar v(r))\mbox{ with }s\in [0,\pi/\sqrt{2}]\mbox{ and }r\in \mathbb R$$
form a smooth family of minimizing geodesics connecting $p$ and $q$. Denote
$$\kappa=d_g(p,z_0).$$

\begin{lemma}\label{Lem: rotation}
    We have
    $$\bar v(r)=\cos \left(\frac{\sqrt{2}r}{\sin (\sqrt{2}\kappa)}\right)\cdot\bar v(0)+\sin \left(\frac{\sqrt{2}r}{\sin(\sqrt{2}\kappa)}\right)\cdot J\bar v(0).$$
\end{lemma}
\begin{proof}
Notice that the Jacobi field
\[
\mathcal J_r(s)=(\mathrm d\exp_p)_{s\bar v(r)}(s\bar v'(r))
\]
along $\bar\gamma_r$ satisfies
\[
   \mathcal J_r(0)=0\mbox{ and }\mathcal J_r(\kappa)=J\bar\gamma_r'(\kappa).
\]
On the other hand, from {Remark \ref{Lem: curvature rigidity remark}} we can construct the Jacobi field
\[
\tilde{\mathcal J}_r(s)=\frac{\sin (\sqrt{2}s)}{\sin(\sqrt{2}\kappa)}J\bar\gamma_r'(s)
\]
along $\bar\gamma_r$ satisfying the same boundary condition. Since $\bar\gamma_r$ are minimizing geodesics, no conjugate point appears along $\bar\gamma_r$ for $0<\kappa<\pi/\sqrt{2}$ and so we have $\mathcal J_r=\tilde{\mathcal J_r}$. In particular, we see
\[
\bar v'(r)=\mathcal J_r'(0)=\tilde{ \mathcal J}_r'(0)=\frac{\sqrt{2}\cdot J\bar v(r)}{\sin(\sqrt{2}\kappa)}.
\]
The result follows from solving this equation.
\end{proof}

Denote
$$L:=\bigcup_{r\in \mathbb R}\bar\gamma_r\big([0,\pi/\sqrt{2}]\big).$$

\begin{lemma}\label{Lem: sphere metric}
    $L$ is a totally-geodesic embedded surface in $(M,g)$ which is isometric to the sphere with diameter $\blue \pi/\sqrt{2}$, where $p$ and $q$ are a pair of antipodal points on $L$.
\end{lemma}
\begin{proof}
    Let $D_{\pi/\sqrt{2},p}$ denote the open tangent disk with radius $\pi/\sqrt{2}$ lying in the tangent plane in $T_pM$ generated by $\bar v(0)$ and $J\bar v(0)$. From Lemma \ref{Lem: rotation}, we know that the map
    $$\exp_p:D_{\pi/\sqrt{2},p}\to L\setminus\{q\}$$ is a homeomorphism. Interchanging $p$ and $q$, there is an open tangent disk $D_{\pi/\sqrt{2},q}\subseteq T_qM$ with radius $\pi/\sqrt{2}$ such that $\exp_q:D_{\pi/\sqrt{2},q}\to L\setminus\{p\}$ is also a homeomorphism. These maps induce a natural smooth structure on $L$, which makes $L$ an (real) embedded surface in $(M^n,g)$.

We now show that $L$ is totally-geodesic. From the smoothness it suffices to work on $L\setminus\{p,q\}$. From the proof of Lemma \ref{Lem: rotation} there is an orthonormal frame $\{e_1,e_2\}$ on $L\setminus\{p,q\}$ given by $e_1=\bar\gamma_r'$ and $e_2=Je_1$. Using the facts $\nabla_{e_1}e_1=0$ and $\nabla J=0$ as well as the symmetry of the Levi-Civita connection, we obtain
\[
\nabla_{e_1}e_2 = 0, \ \ \
\nabla_{e_2}e_1 = -[e_{1},e_{2}], \ \ \
\nabla_{e_2}e_2 = J\nabla_{e_2}e_1 = -J[e_{1},e_{2}].
\]
Then for any vector fields $X=f_1e_1+f_2e_2$ and $Y=g_1e_1+g_2e_2$ on $L$, we compute
\[
\begin{split}\nabla_XY
= {} & X(g_1)e_1+X(g_2)e_2+g_1\nabla_X e_1+g_2\nabla_X e_2 \\
= {} & X(g_1)e_1+X(g_2)e_2
+g_1f_1\nabla_{e_1} e_1 +g_1 f_2 \nabla_{e_2}e_1+ g_2f_1\nabla_{e_1}e_2+g_2f_2\nabla_{e_2}e_2 \\
= {} & X(g_1)e_1+X(g_2)e_2 -g_1 f_2 [e_1,e_2]-g_2f_2J [e_1,e_2].
\end{split}
\]
This shows that the right hand side belongs to $TL$ and thus $L$ is totally geodesic.

Notice that $L$ is simply-connected from the Van-Kampen theorem and has Ricci curvature constantly equals to $2$ from the Gauss equation and Lemma \ref{Lem: curvature rigidity}. Therefore, $L$ is isometric to the sphere with diameter $\blue \pi/\sqrt{2}$ as a Riemannian manifold. Clearly, $p$ and $q$ are a pair of antipodal points on $L$.
\end{proof}

\begin{lemma}\label{Lem: CP1 structure}
    $L$ is a complex curve in $(M,J)$ biholomorphic to $\CP^{1}$.
\end{lemma}
\begin{proof}
    Clearly, $J$ induced from $M$ is a complex structure on $L$ and so $L$ is a complex curve. The simply-connectedness yields that $L$ is biholomorphic to $\CP^1$.
\end{proof}

Now  Proposition~\ref{totally gedesic CP1} follows.
\begin{proof}[Proof of Proposition \ref{totally gedesic CP1}]
    It follows from Lemma \ref{Lem: sphere metric} and Lemma \ref{Lem: CP1 structure}.
\end{proof}

\section{Proof of main theorems}\label{Sec: main theorems}
In this section, we will prove our main results. We first consider the diameter rigidity on $\mathbb{CP}^{n}$ under an additional condition on the volume, i.e. Theorem~\ref{Thm: main1}.

\begin{proof}[Proof of Theorem \ref{Thm: main1}]
The proof is based on combining the method in \cite{LY18}  with our construction of extreme $\CP^1$ from Proposition~\ref{totally gedesic CP1}. It follows from Lemma \ref{Lem: curvature rigidity} that
    $$\diam(\CP^n,\omega)\leq \frac{\pi}{\sqrt 2}.$$
When the equality holds, there exist two points $p$ and $q$ which are connected by a minimizing geodesic with length $\pi/\sqrt 2$. It follows from Proposition \ref{totally gedesic CP1} that there is a holomorphically embedded $\CP^1\subset \CP^n$ whose area is equal to
\[
\vol(\mathbb{S}^{2},g_{\mathbb{S}^{2}}/2) = \vol(\CP^1,\omega_{\CP^1}).
\]
Let $k$ denote its degree, where $k$ is a positive integer. Then we know $[\omega]=k^{-1}[\omega_{\CP^n}]$ and so
    $$\vol(\CP^n,\omega)=k^{-n}\cdot\vol(\CP^n,\omega_{\CP^n}).$$
    From our assumption, we see $k=1$ and so
    $$[\omega]=[\omega_{\CP^n}].$$
Combining with the well-known fact
    $$[\Ric(\omega)]=c_1(\CP^n)=(n+1)[\omega_{\CP^n}],$$
    we compute
    \[
\int_{\mathbb{CP}^n}\mathrm{R}\cdot\omega^{n}
= n\int_{\mathbb{CP}^n}\Ric(\omega)\wedge\omega^{n-1}
= n(n+1)\cdot\vol(\mathbb{CP}^n,\omega_{\mathbb{CP}^n}).
\]
It follows from Lemma \ref{HSC R} that the holomorphic sectional curvature lower bound $\HSC\geq 2$ implies the scalar curvature lower bound $\mathrm{R}\geq n(n+1)$, and so we have $\mathrm{R}\equiv n(n+1)\mbox{ and }\HSC\equiv 2$.
As a consequence, $(\mathbb{CP}^n,\omega)$ is biholomorphically isometric to $(\mathbb{CP}^n,\omega_{\mathbb{CP}^n})$.
\end{proof}

\begin{proof}[Proof of Theorem \ref{Thm: main2}]
It follows from Lemma \ref{Lem: curvature rigidity} that
    $$d(p,q)=\frac{\pi}{\sqrt 2}\mbox{ for any }q\in Q.$$
In particular, there is a minimizing geodesic connecting $p$ and each $q\in Q$ with length $\pi/\sqrt 2$. It follows from Proposition \ref{totally gedesic CP1} that there is an embedded surface $L_q$ consisting of minimizing geodesic connecting $p$ and $q$. Due to $d_g(p,q)=d_g(p,Q)$, these geodesics are also minimizing geodesics connecting $p$ and $Q$. In particular, we have the orthogonal decomposition
    \begin{equation}\label{Eq: tangent space decomposition}
        T_qM=T_qQ\oplus T_qL_q.
    \end{equation}

    We claim that any point $z\in M$ outside $\{p\}\cup Q$ must lie in a minimizing geodesic connecting $p$ and $Q$. To see this, we take $\gamma_z:[0,\ell]\to (M,\omega)$ to be a minimizing geodesic connecting $z$ and $Q$. Clearly, it determines a vector $\gamma_z'(\ell)$ which is orthogonal to $T_{\gamma_z(\ell)}Q$ at $\gamma_z(\ell)$. From the decomposition \eqref{Eq: tangent space decomposition}, we can find a minimizing geodesic on $L_{\gamma_z(\ell)}$ connecting $p$ and $\gamma_z(\ell)$, which has the speed $\gamma_z'(\ell)$ at $\gamma_z(\ell)$. The claim follows from the uniqueness of geodesics.

    Consider the distance functions
    \[
r_{p}(x) = d_g(x,p) \mbox{ and } r_{Q}(x) = d_g(x,Q).
\]
From previous discussion we conclude
\[
\mathrm{Cut}(p) = Q, \ \ \mathrm{Cut}(Q) = \{p\},
\]
and
\begin{equation}\label{Eq: constant sum}
    r_{p} + r_{Q} \equiv \frac{\pi}{\sqrt{2}} \mbox{ in }M.
\end{equation}
On the other hand, it follows from Theorem \ref{Tam-Yu comparison inequality} and Lemma \ref{Lem: curvature rigidity} that we have
\begin{equation}\label{comparison rp}
\begin{split}
(r_{p})_{\alpha\ov{\beta}} \leq {} & \frac{1}{\sqrt{2}}\cot\left(\frac{r_{p}}{\sqrt{2}}\right) g_{\alpha\ov{\beta}} \\[1.6mm]
& +\sqrt{2}\left(\cot(\sqrt{2}r_{p})-\cot\left(\frac{r_{p}}{\sqrt{2}}\right)\right)(r_{p})_{\alpha}(r_{p})_{\ov{\beta}}
\end{split}
\end{equation}
and
\begin{equation}\label{comparison rQ}
\begin{split}
(r_{Q})_{\alpha\ov{\beta}} \leq {} & \frac{1}{\sqrt{2}}\cot\left(\frac{r_{Q}}{\sqrt{2}}\right)(g_{\alpha\ov{\beta}}-g_{\alpha\ov{\beta}}^{Q}) \\[1.6mm]
& +\sqrt{2}\left(\cot(\sqrt{2}r_{Q})-\cot\left(\frac{r_{Q}}{\sqrt{2}}\right)\right)(r_{Q})_{\alpha}(r_{Q})_{\ov{\beta}} \\
& -\frac{1}{\sqrt{2}}\tan\left(\frac{r_{Q}}{\sqrt{2}}\right)g_{\alpha\ov{\beta}}^{Q}
\end{split}
\end{equation}
in $M\setminus(\{p\}\cup Q)$. In particular,
\begin{equation}\label{comparison-combine}
\begin{split}
    g^{\alpha\bar\beta}(r_p+r_Q)_{\alpha\bar\beta}\leq \frac{n-1}{\sqrt 2}&\cot\left(\frac{r_p}{\sqrt 2}\right)+\frac{1}{\sqrt 2}\cot\left(\sqrt 2r_p\right)\\[1.6mm]
    &+\frac{1}{\sqrt 2}\cot\left(\sqrt 2r_Q\right)-\frac{n-1}{\sqrt 2}\tan\left(\frac{r_Q}{\sqrt 2}\right).
\end{split}
\end{equation}
Combining \eqref{comparison-combine} with \eqref{Eq: constant sum}, we obtain
$$0=g^{\alpha\bar\beta}(r_p+r_Q)_{\alpha\bar\beta}\leq 0.$$
This shows that both \eqref{comparison rp} and \eqref{comparison rQ} are equalities.  Now it follows from Theorem \ref{Tam-Yu comparison equality} with $S=\{p\}$ that $$M\setminus Q=B_{\pi/\sqrt{2}}(p)$$ is holomorphically isometric to the geodesic $(\pi/\sqrt{2})$-ball  in $(\mathbb{CP}^{n},\omega_{\CP^n})$. Since $\dim_{\mathbb{C}}Q=n-1$, we have $M=\overline{B_{\pi/\sqrt 2}(p)}$ and so $(M,\omega)$ has $\HSC\equiv2$, yielding that $(M,\omega)$ is biholomorphically isometric to $(\mathbb{CP}^n,\omega_{\mathbb{CP}^n})$.
\end{proof}

\appendix

\section{Preliminary results}

In this section, we collect some preliminary results from previous works.

\subsection{Holomorphic sectional curvature}
\begin{lemma}\label{HSC R}
Let $(M,\omega)$ be a compact $n$-dimensional K\"ahler manifold with $\HSC\geq 2$ at $p\in M$. Then the scalar curvature satisfies $\mathrm{R}\geq n(n+1)$ at $p$, where the equality holds if and only if $\HSC\equiv 2$ at $p$.
\end{lemma}
\begin{proof}
This follows from Berger's averaging trick:
\[
\mathrm{R}(p) = \frac{n(n+1)}{2}\fint_{u\in T_{p}^{1,0}M,\,|u|=1}\HSC(u)\, \mathrm d\theta(u).
\]
\end{proof}

\begin{lemma}[Yang \cite{Yang17}]\label{partial BK}
Let $(M,\omega)$ be a K\"ahler manifold and $p\in M$. Suppose that $e_1$ is a unitary $(1,0)$-vector at $p$ minimizing the holomorphic sectional curvature, i.e.
\[
\mathrm{HSC}(e_{1}) = \min_{u\in T^{1,0}_pM,\,|u|=1}\mathrm{HSC}(u).
\]
Then for any unitary $(1,0)$-vector $u$ at $p$, we have
\[
\mathrm{BK}(e_{1},u)
\geq \frac{\HSC(e_{1})}{2} \cdot (1+|\langle e_{1},\ov{u}\rangle|^{2}).
\]
\end{lemma}
\begin{proof}
This is {\cite[Lemma 4.1]{Yang17}}.
\end{proof}

\subsection{Bisectional curvature}
In this subsection, we collect the complex Hessian comparison theorem proved by Tam-Yu \cite{TY12}. Although the original theorems were stated under the assumption $\BK\geq K$ (we shall take $K=1$ for our use), one can check directly that the arguments still work under the weaker assumption \eqref{radial BK geq K} below.
\begin{theorem}[Tam-Yu \cite{TY12}]\label{Tam-Yu comparison inequality}
Let $(M,\omega)$ be a compact $n$-dimensional K\"ahler manifold and $S$ be a closed complex submanifold. Denote the cut-locus of $S$ by $\mathrm{Cut}(S)$. Set
\[
r(x) = d_g(x,S), \ \ e_{1} = \frac{1}{\sqrt{2}}(\nabla r-\sqrt{-1}J\nabla r).
\]
Suppose that
\begin{equation}\label{radial BK geq K}
\mathrm{BK}(e_{1},u) \geq 1+|\langle e_{1},\ov{u}\rangle|^{2}
\end{equation}
holds for any point $x\in M\setminus \mathrm{Cut}(S)$ and unitary $(1,0)$-vector $u$.
Then we have
\begin{equation}\label{Tam-Yu comparison inequality eqn}
r_{\alpha\ov{\beta}} \leq F(g_{\alpha\ov{\beta}}-g_{\alpha\ov{\beta}}^{S})+Gr_{\alpha}r_{\ov{\beta}}+Hg_{\alpha\ov{\beta}}^{S},
\end{equation}
in $M\setminus \mathrm{Cut}(S)$,
where $g^{S}$ denotes the metric of $S$ parallel transported along the geodesics orthogonal to $S$,
\[
F=\frac{1}{\sqrt 2}\cot\left(\frac{r}{\sqrt 2}\right),\ \ G=\sqrt 2\left(\cot(\sqrt 2r)-\cot\left(\frac{r}{\sqrt 2}\right)\right),
\]
and
\[
H=-\frac{1}{\sqrt 2}\tan\left(\frac{r}{\sqrt 2}\right).
\]
\end{theorem}
\begin{proof}
    This follows from the proof of  \cite[Theorem 2.1]{TY12}.
\end{proof}

\begin{theorem}[Tam-Yu \cite{TY12}]\label{Tam-Yu comparison equality}
If $S=\{p\}$ is a point and \eqref{Tam-Yu comparison inequality eqn} attains equality in Theorem \ref{Tam-Yu comparison inequality}, then any geodesic ball $B_r(p)\subset M\setminus\mathrm{Cut}(S)$ with $r\leq \pi/\sqrt 2$ is holomorphically isometric to the geodesic $r$-ball in $(\CP^n,\omega_{\CP^n})$.
\end{theorem}
\begin{proof}
    This follows from the proof of \cite[Theorem 2.2]{TY12}.
\end{proof}

\subsection{Conformal deformation}\label{append-eg}
The following concerns the relation of the Levi-Civita connection and curvature of a conformally K\"ahler metric with its reference K\"ahler metric (in its conformal class).
\begin{lemma}\label{quantity after conformal}
    Let $(M,g,J)$ be a K\"ahler manifold and $\varphi$ be a smooth function on $M$. Denote
    $$\hat g=e^{2\varphi}g.$$
    Let $\nabla$ and $\widehat\nabla$ be the Levi-Civita connection of metrics $g$ and $\hat g$ respectively.
    Then we have
    $$|\widehat \nabla J|_{\hat g}^2\leq Ce^{-2\varphi}|\nabla\varphi|^2_g,$$
    where $\widehat\nabla J$ is the $2$-tensor given by
   $$(\widehat\nabla J)(v,w)=\widehat\nabla_w(Jv)-J(\widehat\nabla_wv).$$
    We also have
    $$\widehat K(\sigma)=e^{-2\varphi}\left(K(\sigma)-\mathrm{tr}_\sigma(\nabla^2\varphi)-|\nabla_{\sigma^\bot}\varphi|^2\right),$$
   where $\sigma$ denotes a tangent plane, $\widehat K$ and $K$ denote the sectional curvature of $\hat g$ and $g$ respectively.
\end{lemma}

\begin{proof}
This follows from the direct computation, for example see the proof of \cite[Lemma A.1]{CLZ25}.
\end{proof}

{
\section{Counter-example to question~\ref{Xiong-Yang question}}\label{sec:app}

In this section, we discuss an example showing that irreducibility is a necessary condition for the rigidity Question~\ref{Xiong-Yang question} to hold. For $1\leq m<n$, we consider the complex manifold $M^n:=\mathbb{CP}^m\times \mathbb{CP}^{n-m}$, equipped with the K\"ahler metric $\omega:=\omega_{\mathbb{CP}^{m}}+\omega_{\mathbb{CP}^{n-m}}$, where $\omega_{\mathbb{CP}^{k}}$ denotes the Fubini-Study metric on $\mathbb{CP}^{k}$. It is well-known that
\[
\HSC(\mathbb{CP}^{k},\omega_{\mathbb{CP}^{k}}) \equiv 2, \ \ \
\diam(\mathbb{CP}^{k},\omega_{\mathbb{CP}^{k}}) = \frac{\pi}{\sqrt{2}}, \ \ \
\Vol(\mathbb{CP}^{k},\omega_{\mathbb{CP}^{k}}) = \frac{(2\pi)^{k}}{k!}.
\]
Thanks to the product structure, we have 
\begin{equation*}
    \mathrm{diam}(M,\omega)=  \sqrt{\mathrm{diam}^2(\mathbb{CP}^m,\omega_{\mathbb{CP}^{m}})+\mathrm{diam}^2(\mathbb{CP}^{n-m},\omega_{\mathbb{CP}^{n-m}})}=\pi
\end{equation*}
and
\[
\Vol(M,\omega) = \Vol(\mathbb{CP}^{m},\omega_{\mathbb{CP}^{m}})\cdot \Vol(\mathbb{CP}^{n-m},\omega_{\mathbb{CP}^{n-m}})
= \frac{(2\pi)^{n}}{m!(n-m)!}.
\]
Next we examine its holomorphic sectional curvature. Let $u:=(u_{1},u_{2})$ be a unitary $(1,0)$-vector in $(M,\omega)$. We denote $|u_1|^2=s\in[0,1]$ and $|u_2|^2=1-s$ so that
\begin{equation*}
    \begin{split}
        R(u,\ov{u},u,\ov{u}) 
        = {} & R_{\omega_{\mathbb{CP}^{m}}}(u_1,\ov{u}_1,u_1,\ov{u}_1)
        +R_{\omega_{\mathbb{CP}^{n-m}}}(u_2,\ov{u}_2,u_2,\ov{u}_2) \\[1mm]
        = {} & 2s^2+2(1-s)^2\geq 1,
    \end{split}
\end{equation*}
which implies $\HSC(M,\omega)\geq1$. Then $(M,\hat{\omega})$ where $\hat{\omega}:=\omega/2$ satisfies
\[
\HSC(M,\hat\omega) \geq 2, \ \ \
\diam(M,\hat\omega) = \frac{\pi}{\sqrt{2}}, \ \ \
\Vol(M,\hat\omega) = \frac{\pi^{n}}{m!(n-m)!}.
\]
However, $(M,\hat\omega)$ is not biholomorphically isometric to $(\mathbb{CP}^n,\omega_{\mathbb{CP}^n})$. On the other hand, note that
\begin{equation}\label{app eqn}
\vol(M,\hat\omega) \neq k^{-n}\cdot\vol(\mathbb{CP}^n,\omega_{\mathbb{CP}^n}) \mbox{ for all integers }k\geq 2,
\end{equation}
since
\[
k^{n}\cdot\frac{\vol(M,\hat\omega)}{\vol(\mathbb{CP}^n,\omega_{\mathbb{CP}^n})}
= \left(\frac{k}{2}\right)^{n}\binom{n}{m}
\geq \binom{n}{m} > 1.
\]
Furthermore, \eqref{app eqn} does not hold even for $k=1$. Indeed, it follows from $\sum_{k=0}^{n}\binom{n}{k}=2^n$ that $\binom{n}{m}<2^n$ and so $\vol(M,\hat\omega) \neq \vol(\mathbb{CP}^n,\omega_{\mathbb{CP}^n})$.

}


\begin{thebibliography}{99}

\bibitem{CLZ25} Chu, J.; Lee, M.-C.; Zhu, J. {\em On {K}\"ahler manifolds with non-negative mixed curvature}, J. Reine Angew. Math. {\bf 827} (2025), 313--338.

\bibitem{CWZ25} Chu, J.; Wang, F.; Zhang, K., {\em The rigidity of eigenvalues on {K}\"ahler manifolds with positive {R}icci lower bound}, J. Reine Angew. Math. {\bf 820} (2025), 213--233.

\bibitem{DS23} Datar, V.; Seshadri, H. {\em Diameter rigidity for K\"ahler manifolds with positive bisectional curvature}, Math. Ann. {\bf 385} (2023), no. 1--2, 471--479.

\bibitem{Ehr74} Ehrlich, P. E. {\em Continuity properties of the injective radius function}, Compositio Math. {\bf 29} (1974), 151--178.

\bibitem{FXY25} Fang, J.; Xiong, Z.; Yang, X. {\em Laplacian comparison theorems on complete K\"ahler manifolds and applications}, preprint, arXiv:2510.01548.

\bibitem{Frankel61} Frankel, T. {\em Manifolds with positive curvature}, Pacific J. Math. {\bf 11} (1961), 165--174.

\bibitem{Hartshorne70} Hartshorne, R., {\em Ample subvarieties of algebraic varieties}, Notes written in collaboration with C. Musili. Lecture Notes in Math., Vol. 156. Springer-Verlag, Berlin-New York, 1970. xiv+256 pp.

\bibitem{LW05} Li, P.; Wang, J. {\em Comparison theorem for K\"ahler manifolds and positivity of spectrum}, J. Differential Geom. {\bf 69} (2005), no. 1, 43--74.

\bibitem{LY18} Liu, G.; Yuan, Y. {\em Diameter rigidity for K\"ahler manifolds with positive bisectional curvature}, Math. Z. {\bf 290} (2018), no. 3--4, 1055--1061.

\bibitem{Lott21} Lott, J. {\em Comparison geometry of holomorphic bisectional curvature for K\"ahler manifolds and limit spaces}, Duke Math. J. {\bf 170} (2021), no. 14, 3039--3071.

\bibitem{Mori79} Mori, S. {\em Projective manifolds with ample tangent bundles}, Ann. of Math. (2) {\bf 110} (1979), no. 3, 593--606.

\bibitem{SY80} Siu, Y.-T.; Yau, S.-T. {\em Compact K\"ahler manifolds of positive bisectional curvature}, Invent. Math. {\bf 59} (1980), no. 2, 189--204.

\bibitem{TY12} Tam, L.-F.; Yu, C. {\em Some comparison theorems for K\"ahler manifolds}, Manuscripta Math. {\bf 137} (2012), no. 3--4, 483--495.

\bibitem{Tsukamoto57} Tsukamoto, Y. {\em On K\"ahlerian manifolds with positive holomorphic sectional curvature}, Proc. Japan Acad. {\bf 33} (1957), 333--335.

\bibitem{WY25} Wang, M.; Yang, X. {\em First eigenvalue estimates on complete K\"ahler manifolds}, preprint, arXiv:2507.09203.

\bibitem{XY24} Xiong, Z.; Yang, X. {\em Conjugate radius, volume comparison and rigidity}, preprint, arXiv:2408.02080.

\bibitem{Yang17} Yang, X. {\em Big vector bundles and complex manifolds with semi-positive tangent bundles}, Math. Ann. {\bf 367} (2017), no. 1--2, 251--282.

\bibitem{Zhang22} Zhang, K. {\em On the optimal volume upper bound for K\"ahler manifolds with positive Ricci curvature (with an appendix by Yuchen Liu)}, Int. Math. Res. Not. IMRN 2022, no. 8, 6135--6156.

\end{thebibliography}
\end{document}